\documentclass[11pt,a4paper]{article}

\usepackage{amsmath}
\usepackage{amsthm}
\usepackage{amssymb}
\usepackage{amscd}
\usepackage[all]{xy}

\title{Strongly Liftable Schemes and the Kawamata-Viehweg Vanishing
in Positive Characteristic III
\footnote{This paper was partially supported by the National Natural Science Foundation of China
(Grant No.\ 11231003 and 11271070), and the Scientific Research Foundation for the Returned Overseas
Chinese Scholars, State Education Ministry.}}
\author{Qihong Xie \and Jian Wu}
\date{Dedicated to Professor Yujiro Kawamata for his sixtieth birthday}
\theoremstyle{plain}
\newtheorem{prop}{Proposition}[section]
\newtheorem{lem}[prop]{Lemma}
\newtheorem{thm}[prop]{Theorem}
\newtheorem{cor}[prop]{Corollary}

\theoremstyle{definition}
\newtheorem{defn}[prop]{Definition}
\newtheorem*{ack}{Acknowledgments}
\newtheorem*{nota}{Notation}

\theoremstyle{remark}
\newtheorem{rem}[prop]{Remark}
\newtheorem{ex}[prop]{Example}

\newcommand{\Q}{\mathbb Q}

\newcommand{\Z}{\mathbb Z}
\newcommand{\N}{\mathbb N}
\newcommand{\F}{\mathbb F}
\newcommand{\A}{\mathbb A}
\newcommand{\PP}{\mathbb P}
\newcommand{\OO}{\mathcal O}
\newcommand{\II}{\mathcal I}

\newcommand{\LL}{\mathcal L}

\newcommand{\TT}{\mathcal T}

\newcommand{\Supp}{\mathop{\rm Supp}\nolimits}
\newcommand{\Sing}{\mathop{\rm Sing}\nolimits}

\newcommand{\spec}{\mathop{\rm Spec}\nolimits}

\newcommand{\Gal}{\mathop{\rm Gal}\nolimits}
\newcommand{\tor}{\mathop{\rm Tor}\nolimits}

\newcommand{\divisor}{\mathop{\rm div}\nolimits}

\newcommand{\ra}{\rightarrow}

\newcommand{\inj}{\hookrightarrow}
\newcommand{\wt}{\widetilde}

\newcommand{\al}{\alpha}

\begin{document}

\maketitle

\begin{abstract}
A smooth scheme $X$ over a field $k$ of positive characteristic
is said to be strongly liftable over $W_2(k)$, if $X$ and all prime
divisors on $X$ can be lifted simultaneously over $W_2(k)$.
In this paper, we first deduce the Kummer covering trick over $W_2(k)$,
which can be used to construct a large class of smooth projective
varieties liftable over $W_2(k)$, and to give a direct proof of
the Kawamata-Viehweg vanishing theorem on strongly liftable schemes.
Secondly, we generalize almost all of the results in \cite{xie10,xie11}
to the case where everything is considered over $W(k)$, the ring of
Witt vectors of $k$.
\end{abstract}

\setcounter{section}{0}
\section{Introduction}\label{S1}

Throughout this paper, we always work over {\it an algebraically
closed field $k$ of characteristic $p>0$} unless otherwise stated.
A smooth scheme $X$ is said to be strongly liftable over $W_2(k)$,
if $X$ and all prime divisors on $X$ can be lifted simultaneously
over $W_2(k)$. This notion was first introduced in \cite{xie10} to
study the Kawamata-Viehweg vanishing theorem in positive characteristic,
furthermore, some examples and properties of strongly liftable schemes
were also given in \cite{xie10,xie11}.

In this paper, we shall continue to study strongly liftable schemes.
First of all, we deduce the Kummer covering trick over $W_2(k)$ by means of the logarithmic techniques
developed by K. Fujiwara, K. Kato and C. Nakayama. The Kummer covering trick over $W_2(k)$ can be used
to construct a large class of smooth projective varieties liftable over $W_2(k)$
(see \S \ref{S3} for more details).

\begin{thm}\label{1.1}
Let $X$ be a smooth projective variety strongly liftable over $W_2(k)$, $D$ a $\Q$-divisor on $X$
such that the fractional part $\langle D\rangle=\sum_{i\in I}\displaystyle{\frac{a_i}{b_i}}D_i$
satisfies that $0<a_i<b_i$, $(a_i,b_i)=1$ and $p\nmid b_i$ for all $i\in I$, and $\sum_{i\in I}D_i$
is simple normal crossing. Then there exists a finite Galois morphism $\tau: Y\ra X$ with Galois group
$G=\Gal(K(Y)/K(X))$, such that $\tau^*D$ is integral and $Y$ is a smooth projective variety liftable
over $W_2(k)$.
\end{thm}

As an application, the Kummer covering trick over $W_2(k)$ gives a direct proof of the Kawamata-Viehweg
vanishing theorem on strongly liftable schemes.

\begin{cor}\label{1.2}
Let $X$ be a strongly liftable smooth projective variety of dimension $d$, and $D$ an ample $\Q$-divisor on $X$
such that $\Supp\langle D\rangle$ is simple normal crossing. Then we have $H^i(X,K_X+\ulcorner D\urcorner)=0$
for any $i>d-\inf(d,p)$.
\end{cor}

Secondly, we generalize almost all of the results in \cite{xie10,xie11} to the case where everything is considered over
$W(k)$, the ring of Witt vectors of $k$. A smooth scheme $X$ is said to be strongly liftable over $W(k)$, if $X$ and all
prime divisors on $X$ can be lifted simultaneously over $W(k)$. The following is the main result in \S \ref{S5}.

\begin{thm}\label{1.3}
The following varieties are strongly liftable over $W(k)$:
\begin{itemize}
\item[(i)] $\A^n_k$, $\PP^n_k$ and a smooth projective curve;
\item[(ii)] a smooth projective variety of Picard number 1 which is a complete intersection in $\PP^n_k$;
\item[(iii)] A smooth projective rational surface;
\item[(iv)] A smooth projective toric variety.
\end{itemize}
\end{thm}

Finally, by means of cyclic covers over toric varieties, we can obtain a large class of smooth projective 
varieties liftable over $W(k)$ (see Theorem \ref{5.12} and Corollary \ref{5.13} for more details).

\begin{cor}\label{1.4}
Let $X$ be a smooth projective toric variety, and $\LL$ an invertible sheaf on $X$.
Let $N$ be a positive integer prime to $p$, and $D$ an effective divisor on $X$ with
$\LL^N=\OO_X(D)$ and $\Sing(D_{\rm red})=\emptyset$.
Let $\pi:Y\ra X$ be the cyclic cover obtained by taking the $N$-th root out of $D$.
Then $Y$ is a smooth projective scheme liftable over $W(k)$.
\end{cor}

In \S \ref{S2}, we will recall some definitions and preliminary results of liftings over $W_2(k)$.
The Kummer covering trick over $W_2(k)$ will be treated in \S \ref{S3}.
We will give the generalizations to $W(k)$ in \S \ref{S5}.
For the necessary notions and results in birational geometry, we refer the reader to \cite{kmm}.

\begin{nota}
We use $[B]=\sum [b_i] B_i$ (resp.\ $\ulcorner B\urcorner=\sum \ulcorner b_i\urcorner B_i$,
$\langle B\rangle=\sum \langle b_i\rangle B_i$) to denote the round-down (resp.\ round-up,
fractional part) of a $\Q$-divisor $B=\sum b_iB_i$, where for a real number $b$,
$[b]:=\max\{ n\in\Z \,|\,n\leq b \}$, $\ulcorner b\urcorner:=-[-b]$ and $\langle b\rangle:=b-[b]$.
We use $\Sing(D_{\rm red})$ (resp.\ $\Supp(D)$) to denote the singular locus of the
reduced part (resp.\ the support) of a divisor $D$. We use $K(X)$ to denote the field of rational
functions of an integral scheme $X$.
\end{nota}

\begin{ack}
We are deeply indebted to Professor Luc Illusie for pointing out a mistake in an earlier version of
this paper and providing a new proof of Theorem \ref{3.1} by means of the logarithmic techniques.
We would also like to thank the referee for giving many useful comments, which make this paper more 
readable.
\end{ack}

\section{Preliminaries}\label{S2}

\begin{defn}\label{2.1}
Let $W_2(k)$ be the ring of Witt vectors of length two of $k$.
Then $W_2(k)$ is flat over $\Z/p^2\Z$, and $W_2(k)\otimes_{\Z/p^2\Z}\F_p=k$.
The following definition \cite[Definition 8.11]{ev} generalizes the definition
\cite[1.6]{di} of liftings of $k$-schemes over $W_2(k)$.

Let $X$ be a noetherian scheme over $k$, and $D=\sum D_i$ a reduced Cartier
divisor on $X$. A lifting of $(X,D)$ over $W_2(k)$ consists of a scheme
$\wt{X}$ and closed subschemes $\wt{D}_i\subset\wt{X}$, all defined and
flat over $W_2(k)$ such that $X=\wt{X}\times_{\spec W_2(k)}\spec k$ and
$D_i=\wt{D}_i\times_{\spec W_2(k)}\spec k$. We write
$\wt{D}=\sum \wt{D}_i$ and say that $(\wt{X},\wt{D})$ is a lifting
of $(X,D)$ over $W_2(k)$, if no confusion is likely.

Let $\LL$ be an invertible sheaf on $X$. A lifting of $(X,\LL)$ consists
of a lifting $\wt{X}$ of $X$ over $W_2(k)$ and an invertible sheaf $\wt{\LL}$
on $\wt{X}$ such that $\wt{\LL}|_X=\LL$. For simplicity, we say that
$\wt{\LL}$ is a lifting of $\LL$ on $\wt{X}$, if no confusion is likely.
\end{defn}

Let $\wt{X}$ be a lifting of $X$ over $W_2(k)$. Then $\OO_{\wt{X}}$ is flat
over $W_2(k)$, hence flat over $\Z/p^2\Z$. Note that there is an exact
sequence of $\Z/p^2\Z$-modules:
\[
0\ra p\cdot\Z/p^2\Z\ra \Z/p^2\Z\stackrel{r}{\ra} \Z/p\Z\ra 0,
\]
and a $\Z/p^2\Z$-module isomorphism $p:\Z/p\Z\ra p\cdot\Z/p^2\Z$.
Tensoring the above by $\OO_{\wt{X}}$, we obtain an exact sequence of
$\OO_{\wt{X}}$-modules:
\begin{eqnarray}
0\ra p\cdot\OO_{\wt{X}}\ra \OO_{\wt{X}}\stackrel{r}{\ra}
\OO_X\ra 0, \label{es1}
\end{eqnarray}
and an $\OO_{\wt{X}}$-module isomorphism
\begin{eqnarray}
p:\OO_X\ra p\cdot\OO_{\wt{X}}, \label{es2}
\end{eqnarray}
where $r$ is the reduction modulo $p$ satisfying $p(x)=p\wt{x}$,
$r(\wt{x})=x$ for $x\in\OO_X$, $\wt{x}\in\OO_{\wt{X}}$.

\begin{defn}\label{2.2}
Let $X$ be a smooth scheme over $k$. $X$ is said to be strongly liftable
over $W_2(k)$, if there is a lifting $\wt{X}$ of $X$ over $W_2(k)$, such that
for any prime divisor $D$ on $X$, $(X,D)$ has a lifting $(\wt{X},\wt{D})$
over $W_2(k)$ as in Definition \ref{2.1}, where $\wt{X}$ is fixed for all
liftings $\wt{D}$.
\end{defn}

Let $X$ be a smooth scheme over $k$, $\wt{X}$ a lifting of $X$ over $W_2(k)$,
$D$ a prime divisor on $X$ and $\LL_D=\OO_X(D)$ the associated invertible
sheaf on $X$. Then there is an exact sequence of abelian sheaves:
\begin{eqnarray}
0\ra \OO_X\stackrel{q}{\ra} \OO^*_{\wt{X}}\stackrel{r}{\ra}\OO^*_X\ra 1,
\label{es3}
\end{eqnarray}
where $q(x)=p(x)+1$ for $x\in\OO_X$, $p:\OO_X\ra p\cdot\OO_{\wt{X}}$ is
the isomorphism (\ref{es2}) and $r$ is the reduction modulo $p$. The exact
sequence (\ref{es3}) gives rise to an exact sequence of cohomology groups:
\begin{eqnarray}
H^1(\wt{X},\OO^*_{\wt{X}})\stackrel{r}{\ra} H^1(X,\OO^*_X)\ra
H^2(X,\OO_X). \label{es4}
\end{eqnarray}
If $r:H^1(\wt{X},\OO^*_{\wt{X}})\ra H^1(X,\OO^*_X)$ is surjective,
then $\LL_D$ has a lifting $\wt{\LL}_D$. We combine (\ref{es1}) and (\ref{es2})
to obtain an exact sequence of $\OO_{\wt{X}}$-modules:
\begin{eqnarray}
0\ra \OO_{X}\stackrel{p}{\ra} \OO_{\wt{X}}\stackrel{r}{\ra}
\OO_X\ra 0. \label{es5}
\end{eqnarray}
Tensoring (\ref{es5}) by $\wt{\LL}_D$, we have an exact sequence of
$\OO_{\wt{X}}$-modules:
\begin{eqnarray*}
0\ra \LL_D\stackrel{p}{\ra} \wt{\LL}_D\stackrel{r}{\ra} \LL_D\ra 0,
\end{eqnarray*}
which gives rise to an exact sequence of cohomology groups:
\begin{eqnarray}
H^0(\wt{X},\wt{\LL}_D)\stackrel{r}{\ra} H^0(X,\LL_D)\ra H^1(X,\LL_D).
\label{es6}
\end{eqnarray}

There is a criterion for strong liftability over $W_2(k)$ \cite[Proposition 2.5]{xie11}.

\begin{prop}\label{2.3}
Let $X$ be a smooth scheme over $k$, and $\wt{X}$ a lifting of $X$ over
$W_2(k)$. If for any prime divisor $D$ on $X$, there is a lifting $\wt{\LL}_D$
of $\LL_D=\OO_X(D)$ on $\wt{X}$ such that the natural map $r:H^0(\wt{X},\wt{\LL}_D)\ra
H^0(X,\LL_D)$ is surjective, then $X$ is strongly liftable over $W_2(k)$.
\end{prop}

\section{Kummer covering trick over $W_2(k)$}\label{S3}

The Kummer covering trick \cite[Theorem 17]{ka81}, due to Kawamata, is a powerful technique
in birational geometry of algebraic varieties. The following theorem, i.e.\ the Kummer covering
trick over $W_2(k)$, is the main result in this section.

\begin{thm}\label{3.1}
Let $X$ be a smooth projective variety strongly liftable over $W_2(k)$, $D$ a $\Q$-divisor on $X$
such that the fractional part $\langle D\rangle=\sum_{i\in I}\displaystyle{\frac{a_i}{b_i}}D_i$
satisfies that $0<a_i<b_i$, $(a_i,b_i)=1$ and $p\nmid b_i$ for all $i\in I$, and $\sum_{i\in I}D_i$
is simple normal crossing. Then there exists a finite Galois morphism $\tau: Y\ra X$ from a smooth
projective variety $Y$ with Galois group $G=\Gal(K(Y)/K(X))$ which satisfies the following conditions:
\begin{itemize}
\item[(i)] $\tau^*D$ becomes an integral divisor on $Y$;
\item[(ii)] $\OO_X([D])\cong (\tau_*\OO_Y(\tau^*D))^G$, $\OO_X(K_X+\ulcorner D\urcorner)\cong (\tau_*\OO_Y(K_Y+\tau^*D))^G$,
where $G$ acts naturally on $\tau_*\OO_Y(\tau^*D)$ and on $\tau_*\OO_Y(K_Y+\tau^*D)$. Via these isomorphisms, $\OO_X([D])$
(resp.\ $\OO_X(K_X+\ulcorner D\urcorner)$) turns out to be a direct summand of $\tau_*\OO_Y(\tau^*D)$ (resp.\ $\tau_*\OO_Y(K_Y+\tau^*D)$); and
\item[(iii)] $Y$ is liftable over $W_2(k)$.
\end{itemize}
\end{thm}

We shall prove the Kummer covering trick and reprove the cyclic cover trick over $W_2(k)$
by means of the logarithmic techniques developed by K. Fujiwara, K. Kato and C. Nakayama.
For the necessary definitions and results for the logarithmic theory, e.g.\ log structure, fs log scheme,
log \'etale morphism, Kummer \'etale cover, log regular and so on, we refer the reader to \cite{kato89,kato94},
\cite{na97} and \cite{il02}. It should be mentioned that almost all of the arguments in this section are
essentially due to L. Illusie.

First of all, we recall \cite[Theorem 7.6]{il02} as follows.

\begin{thm}\label{3.2}
Let $X$ be a noetherian, log regular fs log scheme, with open subset of triviality of the log structure $U$,
and let $D=X-U$. Let $\mathrm{Kcov}(X)$ denote the category of finite Kummer \'etale covers of $X$, and
$\mathrm{Etcovtame}(U)$ the category of classical finite \'etale covers of $U$ which are tamely ramified
along $D$. Then the restriction functor
\begin{eqnarray}
\mathrm{Kcov}(X)\ra \mathrm{Etcovtame}(U) \label{es7}
\end{eqnarray}
is an equivalence of categories.
\end{thm}

In this section, we always focus on a particular case where $X$ is a regular scheme, whose log structure
is defined by a simple normal crossing divisor $D=\sum_{i\in I}D_i$, then Theorem \ref{3.2} is a simple
application of a lemma of Abhyankar. In this case, a quasi-inverse to (\ref{es7}) is obtained by associating
to a classical finite \'etale cover $V$ of $U$, tamely ramified along $D$, the normalization $Y$ of $X$ in $V$.

\begin{thm}\label{3.3}
Let $X\subset \wt{X}$ be an exact closed immersion of fs log schemes, given by an ideal sheaf $\II$ of square
zero. Let $f:Y\ra X$ be a log \'etale morphism of fs log schemes. Then there exists a log \'etale morphism
$\wt{f}:\wt{Y}\ra \wt{X}$, which is a lifting of $f:Y\ra X$ and is unique up to a unique isomorphism.
\end{thm}

\begin{proof}
Let $\omega=\omega^1_{Y/X}$ be the relative sheaf of log differentials of $f$, and $\omega^\vee$ the dual sheaf
of $\omega$. By \cite[Proposition 3.14]{kato89}, there is an obstruction in $H^2(Y,\omega^\vee\otimes\II)$ to
the existence of such a lifting, moreover isomorphism classes of liftings form an affine space under
$H^1(Y,\omega^\vee\otimes\II)$ and the automorphism group of a lifting is $H^0(Y,\omega^\vee\otimes\II)$.
By assumption, $f:Y\ra X$ is log \'etale, hence $\omega=0$, which implies that the above three cohomology groups
are all zero.
\end{proof}

\begin{lem}\label{3.4}
With notation and assumptions as in Theorem \ref{3.3}, if furthermore $f:Y\ra X$ is a finite Kummer \'etale cover,
then so is $\wt{f}:\wt{Y}\ra \wt{X}$.
\end{lem}

\begin{proof}
Locally for the \'etale topology, any Kummer \'etale cover is isomorphic to a standard Kummer \'etale cover
(cf.\ \cite[3.1.4]{st02}). Therefore, \'etale locally around a geometric point $x$ of $X$, there is a chart
$P\ra M_X$ and a Kummer morphism $P\ra Q$ of fs monoids with $Q^{gp}/P^{gp}$ annihilated by an integer $n$
invertible on $X$, such that $f:Y\ra X$ is the fs pullback of the morphism
$\spec\Z[Q]\ra \spec\Z[P]$ (cf.\ \cite[Proposition 3.2 and Definition 3.5]{il02}):
\[
\xymatrix{
Y \ar[d]_{f} \ar[r] & \spec\Z[Q] \ar[d] \\
X \ar[r] & \spec\Z[P].
}
\]
Since $X\subset \wt{X}$ is an exact closed immersion and the lifting $\wt{f}:\wt{Y}\ra \wt{X}$ is unique,
there is a lifting $P\ra M_{\wt{X}}$ such that $\wt{f}:\wt{Y}\ra \wt{X}$ is the fs pullback
of the same morphism by the lifted chart $\wt{X}\ra \spec\Z[P]$:
\[
\xymatrix{
\wt{Y} \ar[d]_{\wt{f}} \ar[r] & \spec\Z[Q] \ar[d] \\
\wt{X} \ar[r] & \spec\Z[P].
}
\]
Thus $\wt{f}:\wt{Y}\ra \wt{X}$ is also a finite Kummer \'etale cover.
\end{proof}

\begin{lem}\label{3.5}
Let $X$ be a regular scheme, and $\LL$ an invertible sheaf on $X$. Let $n$ be a positive integer prime to $p$,
and $\sum_{i\in I}D_i$ a simple normal crossing divisor on $X$ such that $\LL^n\cong \OO_X(\sum_{i\in I}a_iD_i)$
for some integers $a_i>0$. Then the associated cyclic cover $f:Y\ra X$ is a finite Kummer \'etale cover
with Galois group $\mu_n$.
\end{lem}

\begin{proof}
\'Etale locally around a geometric point $x$ of $X$, $f:Y\ra X$ is given by a cartesian square of fs log schemes:
\[
\xymatrix{
Y \ar[d]_{f} \ar[r] & \spec\Z[Q] \ar[d] \\
X \ar[r] & \spec\Z[\N^r],
}
\]
where the lower morphism is induced by a chart $\N^r\ra M_X$ sending $e_i\in\N^r$ to $t_i$, where $(t_i)_{i\in J}$
is a set of local equations of the branches of $\sum_{i\in I}D_i$ passing through $x$, and $Q$ is given by a
cocartesian square of fs monoids:
\[
\xymatrix{
\N \ar[r] & Q \\
\N \ar[u]^{n\cdot} \ar[r]^\sigma & \N^r, \ar[u]
}
\]
where $\sigma(1)=(a_i)_{i\in J}$, i.e.\ $Q$ is the saturation of the amalgamated sum of $l\mapsto nl$ and $\sigma$.
Indeed, denote by $f':Y'\ra X$ the fs pullback of $\spec\Z[Q]\ra \spec\Z[\N^r]$ by the morphism $X\ra \spec \Z[\N^r]$.
Then $f':Y'\ra X$ is a Kummer \'etale cover by \cite[Proposition 3.2]{il02}. Since $X$ is log regular and $f':Y'\ra X$
is log \'etale, $Y'$ is also log regular, hence is normal by \cite[7.3(c,e)]{il02}. Over the complement of
$\sum_{i\in I}D_i$, $Y'$ is the classical $\mu_n$-\'etale cover defined by the equation $z^n=\prod_{i\in J}t_i^{a_i}$,
so $Y'$ is just the normalization of $X$ in this classical \'etale cover, which coincides with the definition of the
cyclic cover $Y$ (cf. \cite[Chapter 3]{ev}).
\end{proof}

\begin{thm}\label{3.6}
Let $X$ be a smooth scheme over $S=\spec k$, and $\LL$ an invertible sheaf on $X$. Let $n$ be a positive integer
prime to $p$, and $\sum_{i\in I}D_i$ a simple normal crossing divisor on $X$ such that $\LL^n\cong
\OO_X(\sum_{i\in I}a_iD_i)$ for some integers $a_i>0$. Assume that $(X,\sum_{i\in I}D_i)$ has a lifting
$(\wt{X},\sum_{i\in I}\wt{D}_i)$ over $\wt{S}=\spec W_2(k)$. Then the associated cyclic cover $f:Y\ra X$ is
liftable over $\wt{S}$.
\end{thm}

\begin{proof}
Put the trivial log structure on both $S$ and $\wt{S}$, and put the log structure $M_{\wt{X}}$ on $\wt{X}$ given by
the family $\{\OO_{\wt{X}}(\wt{D}_i),i\in I\}$ (cf.\ \cite[Complement 1]{kato89}). Then $\wt{X}/\wt{S}$ is a log
smooth lifting of $X/S$, having \'etale locally a chart $\N^r\ra M_{\wt{X}}$, whose composition with $M_{\wt{X}}\ra
\OO_{\wt{X}}$ sends $e_i\in\N^r$ to a local equation $\wt{t}_i$ of the lifting $\wt{D}_i$ of $D_i$.
By Theorem \ref{3.3}, Lemmas \ref{3.4} and \ref{3.5}, $f:Y\ra X$ is a finite Kummer \'etale cover, hence globally
lifts to a finite Kummer \'etale cover $\wt{f}:\wt{Y}\ra \wt{X}$, which \'etale locally on $\wt{X}$ is a standard
Kummer \'etale cover modeled on $\N^r\ra Q$ as in Lemma \ref{3.5}. $\wt{X}$ is log smooth over $\wt{S}$ and $\wt{f}:
\wt{Y}\ra \wt{X}$ is log \'etale, hence $\wt{Y}$ is log smooth over $\wt{S}$, in particular is flat (as a scheme)
over $\wt{S}$. Thus $\wt{Y}$ is a lifting of $Y$ over $\wt{S}$.
\end{proof}

From now on, we return to the proof of Theorem \ref{3.1} and first give a similar construction to that in
\cite[Theorem 1-1-1]{kmm}. Let $d=\dim X$. Take a positive integer $m$ such that $m\langle D\rangle$ is integral,
$p\nmid m$ and $m$ is sufficiently large if necessary. Take a very ample divisor $M$ on $X$ such that $mM-D_i$ is
also very ample for all $i\in I$. Take general members $H^{(i)}_k\in |mM-D_i|$ for all $i\in I$ and $1\leq k\leq d$
such that $\Supp\langle D\rangle\cup\Supp(\Sigma_{i,k}H^{(i)}_k)$ is a simple normal crossing divisor.
Let $X=\cup_{\al\in A}U_\al$ be an affine open cover of $X$ which trivializes $M$,
$\{a_{\al\beta}\in H^0(U_\al\cap U_\beta,\OO_X^*)\}$ the transition function of $M$,
and $\{\varphi^{(i)}_{k\al}\in H^0(U_\al,\OO_X)\}$ local sections such that
\[
(H^{(i)}_k+D_i)|_{U_\al}=\divisor(\varphi^{(i)}_{k\al})\,\,\hbox{on}\,\,U_\al\,\,\hbox{and}\,\,
\varphi^{(i)}_{k\al}=a_{\al\beta}^m\varphi^{(i)}_{k\beta}\,\,\hbox{on}\,\,U_\al\cap U_\beta.
\]
Note that $K(X)\big[ (\varphi^{(i)}_{k\al})^{1/m} \big]_{i,k}=K(X)\big[ (\varphi^{(i)}_{k\beta})^{1/m} \big]_{i,k}$
for any $\al,\beta\in A$, so we can take the normalization of $X$ in $K(X)\big[ (\varphi^{(i)}_{k\al})^{1/m} \big]_{i,k}$
for some $\al\in A$, which is denoted by $Y$. Since $K(Y)/K(X)$ is a Kummer extension, the induced morphism
$\tau:Y\ra X$ is a finite Galois morphism.

Take any closed point $x\in U_\al$ and set $I_x=\{i\in I\,|\,x\in D_i\}$. Then for each $i\in I_x$, there exists
$1\leq k_i\leq d$ such that $x\not\in H^{(i)}_{k_i}$. Put
\[
R_x=\{\varphi^{(i)}_{k_i\al}\,|\,i\in I_x\}\cup \{\varphi^{(i)}_{k\al}/\varphi^{(i)}_{k_i\al}\,|\,i\in I_x,x\in H^{(i)}_k\}\cup
\{\varphi^{(i)}_{k\al}\,|\,i\in I\setminus I_x,x\in H^{(i)}_k\}.
\]
Then the equations in $R_x$ form a part of a regular system of parameters of $\OO_{X,x}$. Put
\[
T_x=\{\varphi^{(i)}_{k\al}/\varphi^{(i)}_{k_i\al}\,|\,i\in I_x,x\not\in H^{(i)}_k\}\cup
\{\varphi^{(i)}_{k\al}\,|\,i\in I\setminus I_x,x\not\in H^{(i)}_k\}.
\]
Then any equation $\psi\in T_x$ is a unit in $\OO_{X,x}$. Since $p\nmid m$, by the same argument as in \cite[Lemma 1-1-2]{kmm},
we can prove that $Y$ is nonsingular, hence is smooth over $k$.

\begin{lem}\label{3.7}
The induced morphism $\tau:Y\ra X$ is a finite Kummer \'etale cover.
\end{lem}

\begin{proof}
Put the log structure on $X$ defined by the simple normal crossing divisor $\Supp\langle D\rangle\cup\Supp(\Sigma_{i,k}H^{(i)}_k)$.
Let $t_i$ and $s_{i,k}$ $(i\in I,1\leq k\leq d)$ be the local parameters of $D_i$ and $H^{(i)}_k$ respectively.
\'Etale locally around a geometric point $x$ of $X$, $\tau:Y\ra X$ is given by a cartesian square of fs log schemes:
\[
\xymatrix{
Y \ar[d]_{\tau} \ar[r] & \spec\Z[Q] \ar[d] \\
X \ar[r] & \spec\Z[\N^r],
}
\]
where the lower morphism is induced by a chart $\N^r\ra M_X$ sending $e_i\in\N^r$ to $t_is_{i,k_i}$, $s_{i,k}/s_{i,k_i}$ and $t_is_{i,k}$
respectively, which are described as in $R_x$, and $Q$ is given by a cocartesian square of fs monoids:
\[
\xymatrix{
\N \ar[r] & Q \\
\N \ar[u]^{m\cdot} \ar[r]^\sigma & \N^r, \ar[u]
}
\]
where $\sigma(1)=(1,\cdots,1)$, i.e.\ $Q$ is the saturation of the amalgamated sum of $l\mapsto ml$ and $\sigma$.
Indeed, denote by $\tau':Y'\ra X$ the fs pullback of $\spec\Z[Q]\ra \spec\Z[\N^r]$ by the morphism $X\ra \spec \Z[\N^r]$.
Then $\tau':Y'\ra X$ is a Kummer \'etale cover by \cite[Proposition 3.2]{il02}. Since $X$ is log regular and $\tau':Y'\ra X$
is log \'etale, $Y'$ is also log regular, hence is normal by \cite[7.3(c,e)]{il02}. Over the complement of
$\Supp\langle D\rangle\cup\Supp(\Sigma_{i,k}H^{(i)}_k)$, $Y'$ is the classical \'etale cover defined by the equations
$z^m_{i,k}=t_is_{i,k}$ $(i\in I,1\leq k\leq d)$, so $Y'$ is just the normalization of $X$ in this classical \'etale cover,
which coincides with the construction of $Y$.
\end{proof}

\begin{proof}[Proof of Theorem \ref{3.1}]
Let $\wt{X}$ be a lifting of $X$ over $W_2(k)$ as in Definition \ref{2.2}, $\wt{D}_i$ liftings of $D_i$ and $\wt{H}^{(i)}_k$
liftings of $H^{(i)}_k$ for all $i\in I$ and $1\leq k\leq d$. Put the trivial log structure on both $S=\spec k$ and
$\wt{S}=\spec W_2(k)$, and put the log structure $M_{\wt{X}}$ on $\wt{X}$ given by the family $\{\OO_{\wt{X}}(\wt{D}_i),
\OO_{\wt{X}}(\wt{H}^{(i)}_k),i\in I,1\leq k\leq d\}$ (cf.\ \cite[Complement 1]{kato89}). Then $\wt{X}/\wt{S}$ is a log smooth
lifting of $X/S$. By Theorem \ref{3.3}, Lemmas \ref{3.4} and \ref{3.7}, $\tau:Y\ra X$ is a finite Kummer \'etale cover,
hence globally lifts to a finite Kummer \'etale cover $\wt{\tau}:\wt{Y}\ra \wt{X}$, which \'etale locally on $\wt{X}$
is a standard Kummer \'etale cover modeled on $\N^r\ra Q$ as in Lemma \ref{3.7}. $\wt{X}$ is log smooth over $\wt{S}$ and
$\wt{\tau}:\wt{Y}\ra \wt{X}$ is log \'etale, hence $\wt{Y}$ is log smooth over $\wt{S}$, in particular is flat (as a scheme)
over $\wt{S}$. Thus $\wt{Y}$ is a lifting of $Y$ over $\wt{S}$, which proves the condition (iii). As for the conditions
(i) and (ii), the proof is similar to that in \cite[Theorem 1-1-1]{kmm}.
\end{proof}

\begin{rem}\label{3.8}
The most difficult part in Theorem \ref{3.1} is to prove that $Y$ is liftable over $W_2(k)$. The variety $Y$ was constructed 
affine locally and defined to be the normalization of $X$ in a certain Kummer extension. In fact, liftings $\wt{Y}$ of $Y$ 
can be easily and naturally constructed on affine pieces, and it is a naive method to glue these liftings together to get 
a global lifting $\wt{Y}$ of $Y$. However, we encountered many difficulties in this way, and finally, we had to use the 
logarithmic techniques to yield the required global lifting $\wt{Y}$ of $Y$ in the proof of Theorem \ref{3.1}.
\end{rem}

By means of the Kummer covering trick over $W_2(k)$, we can construct a large class of liftable smooth projective varieties
from strongly liftable schemes. On the other hand, the following corollary, i.e.\ the Kawamata-Viehweg vanishing theorem
\cite{ka82,vi82} on strongly liftable schemes, is a direct consequence of Theorem \ref{3.1}. It should be mentioned that
Corollary \ref{3.9} has already been proved in \cite{ha98} and \cite{mo} by different methods.

\begin{cor}\label{3.9}
Let $X$ be a strongly liftable smooth projective variety of dimension $d$, and $D$ an ample $\Q$-divisor on $X$
such that $\Supp\langle D\rangle$ is simple normal crossing. Then we have $H^i(X,K_X+\ulcorner D\urcorner)=0$
for any $i>d-\inf(d,p)$.
\end{cor}

\begin{proof}
We write $\langle D\rangle=\sum_{i\in I}\displaystyle{\frac{a_i}{b_i}}D_i$ such that $0<a_i<b_i$, $(a_i,b_i)=1$ for all $i\in I$,
and $\sum_{i\in I}D_i$ is simple normal crossing. Since ampleness is an open condition, we can change the coefficients $a_i/b_i$
slightly such that $D$ is still ample, and $p\nmid b_i$ for all $i\in I$, whereas $\ulcorner D\urcorner$ is unchanged.

Let $\tau:Y\ra X$ be a Kummer covering associated to $D$ as in Theorem \ref{3.1}. Then $Y$ is a smooth projective variety which is
liftable over $W_2(k)$. By the Kodaira vanishing theorem \cite[Corollaire 2.8]{di}, we have $H^i(Y,K_Y+\tau^*D)=0$ for any $i>d-\inf(d,p)$.
By Theorem \ref{3.1}(ii), there is an injection $H^i(X,K_X+\ulcorner D\urcorner)\inj H^i(X,\tau_*\OO_Y(K_Y+\tau^*D))=H^i(Y,K_Y+\tau^*D)$,
hence $H^i(X,K_X+\ulcorner D\urcorner)=0$ holds for any $i>d-\inf(d,p)$.
\end{proof}

\section{Strongly liftable schemes over $W(k)$}\label{S5}

In this section, we shall generalize almost all of the results in \cite{xie10,xie11} to the case where everything is
considered over $W(k)$, the ring of Witt vectors of $k$. First of all, we give some definitions and preliminary results.

\begin{defn}\label{5.1}
Denote $W(k)$ the ring of Witt vectors of $k$. The underlying set of $W(k)$ is just $k^\N$, so a Witt vector
is written as $a=(a_0,a_1,\cdots,a_n,\cdots)$, where $a_i\in k$ are called the coordinates of $a$. For
$a,b\in W(k)$, the addition and the multiplication of $a$ and $b$ in $W(k)$ are defined as follows:
\begin{eqnarray*}
a+b &=& (S_0(a,b),S_1(a,b),\cdots,S_n(a,b),\cdots), \\
a\cdot b &=& (P_0(a,b),P_1(a,b),\cdots,P_n(a,b),\cdots),
\end{eqnarray*}
where $S_n$ and $P_n$ are certain polynomials with integral coefficients in the coordinates of $a$ and $b$
with indices $\leq n$:
\begin{eqnarray*}
S_0(a,b) &=& a_0+b_0,\,\,S_1(a,b)=a_1+b_1-\sum_{0<i<p}\frac{1}{p}{p\choose i}a_0^ib_0^{p-i},\cdots, \\
P_0(a,b) &=& a_0b_0,\,\,P_1(a,b)=a_0^pb_1+b_0^pa_1,\cdots.
\end{eqnarray*}
The unit element of $W(k)$ is the vector $(1,0,\cdots,0,\cdots)$.

Let $n\geq 1$ be an integer. The set $k^n$, with the addition (resp.\ multiplication) defined by the polynomials
$S_i$ (resp.\ $P_i$) for $0\leq i\leq n-1$, is a quotient ring of $W(k)$, denoted by $W_n(k)$, called the ring
of Witt vectors of length $n$ of $k$. We have that $W_1(k)=k$ and $W(k)=\varprojlim W_n(k)$ is the inverse limit
of $W_n(k)$ under the following surjective homomorphisms of rings, which are called restriction maps:
\begin{eqnarray*}
R:W_{n+1}(k)\ra W_n(k),\,\,(a_0,a_1,\cdots,a_n)\mapsto (a_0,a_1,\cdots,a_{n-1}).
\end{eqnarray*}
There is a homomorphism of $W(k)$-modules, which is called a shift map:
\begin{eqnarray*}
V:W(k)\ra W(k),\,\,(a_0,a_1,\cdots,a_n,\cdots)\mapsto (0,a_0,\cdots,a_{n-1},\cdots).
\end{eqnarray*}
By passage to the quotient, we have a homomorphism of $W_{n+1}(k)$-modules:
\begin{eqnarray*}
V:W_n(k)\ra W_{n+1}(k),\,\,(a_0,a_1,\cdots,a_{n-1})\mapsto (0,a_0,\cdots,a_{n-1}).
\end{eqnarray*}
We have the following exact sequences for any $n,m\geq 1$:
\begin{eqnarray}
0\ra W(k)\stackrel{V^n}{\ra} W(k)\stackrel{r}{\ra} W_n(k)\ra 0, \label{es14} \\
0\ra W_m(k)\stackrel{V^n}{\ra} W_{n+m}(k)\stackrel{R^m}{\ra} W_n(k)\ra 0. \label{es15}
\end{eqnarray}

We refer the reader to \cite[II, \S6]{se79} for more details.
\end{defn}

Next, we will give some definitions of liftings over $W_*(k)$, which generalize those 
notions as in Definition \ref{2.1}, where $*=$ null or $n\geq 2$ unless otherwise stated.

\begin{defn}\label{5.2}
Let $X$ be a noetherian scheme over $k$, and $D=\sum D_i$ a reduced Cartier divisor on $X$.
A lifting of $(X,D)$ over $W_*(k)$ consists of a scheme $\wt{X}$ and closed subschemes 
$\wt{D}_i\subset\wt{X}$, all defined and flat over $W_*(k)$, such that 
$X=\wt{X}\times_{\spec W_*(k)}\spec k$ and $D_i=\wt{D}_i\times_{\spec W_*(k)}\spec k$. 
We write $\wt{D}=\sum \wt{D}_i$ and say that $(\wt{X},\wt{D})$ is a lifting of $(X,D)$ 
over $W_*(k)$, if no confusion is likely.

Let $\LL$ be an invertible sheaf on $X$. A lifting of $(X,\LL)$ over $W_*(k)$
consists of a lifting $\wt{X}$ of $X$ over $W_*(k)$ and an invertible sheaf
$\wt{\LL}$ on $\wt{X}$ such that $\wt{\LL}|_X=\LL$. For simplicity, we say that
$\wt{\LL}$ is a lifting of $\LL$ on $\wt{X}$, if no confusion is likely.

Let $X$ be a smooth scheme over $k$. We say that $\wt{X}$ is a smooth lifting of $X$
over $W_*(k)$, if $\wt{X}$ is smooth over $W_*(k)$ and $X=\wt{X}\times_{\spec W_*(k)}\spec k$.
\end{defn}

Assume that $\wt{X}$ is a lifting of $X$ over $W(k)$, then $X_n:=\wt{X}\times_{\spec W(k)}\spec W_n(k)$
is a lifting of $X$ over $W_n(k)$ for any $n\geq 2$. Tensoring (\ref{es14}) and (\ref{es15}) by $\OO_{\wt{X}}$,
we have the following exact sequences for any $n,m\geq 1$:
\begin{eqnarray}
0\ra \OO_{\wt{X}}\stackrel{V^n}{\ra} \OO_{\wt{X}}\stackrel{r}{\ra} \OO_{X_n}\ra 0, \label{es16} \\
0\ra \OO_{X_m}\stackrel{V^n}{\ra} \OO_{X_{n+m}}\stackrel{R^m}{\ra} \OO_{X_n}\ra 0. \label{es17}
\end{eqnarray}

Note that if everything is considered over $W_2(k)$, then Definition \ref{5.3} 
coincides with Definition \ref{2.2} by \cite[Lemma 2.2]{xie11}.

\begin{defn}\label{5.3}
Let $X$ be a smooth scheme over $k$. $X$ is said to be strongly liftable
over $W_*(k)$, if there is a smooth lifting $\wt{X}$ of $X$ over $W_*(k)$, 
such that for any prime divisor $D$ on $X$, $(X,D)$ has a lifting $(\wt{X},\wt{D})$
over $W_*(k)$ as in Definition \ref{5.2}, where $\wt{X}$ is fixed for all liftings $\wt{D}$.
\end{defn}

\begin{lem}\label{5.4}
Let $X$ be a smooth scheme over $k$, $D=\sum D_i$ a reduced divisor on $X$, and $Z\subset X$
a closed subscheme smooth over $k$ of codimension $s\geq 2$. Let $\pi:X'\ra X$ be the blow-up of
$X$ along $Z$ with the exceptional divisor $E$, and $D'=\pi_*^{-1}D$ the strict transform of $D$.
If there are schemes $\wt{Z}\subset\wt{X}$, all defined and smooth over $W_*(k)$, and closed 
subschemes $\wt{D}_i\subset\wt{X}$, all defined and flat over $W_*(k)$, such that 
$\wt{X}\times_{\spec W_*(k)}\spec k=X$, $\wt{Z}\times_{\spec W_*(k)}\spec k=Z$ and 
$\wt{D}_i\times_{\spec W_*(k)}\spec k=D_i$, then $(X',D'+E)$ has a lifting over $W_*(k)$.
\end{lem}

\begin{proof}
It is similar to that of \cite[Lemma 2.4]{xie10}.
\end{proof}

\begin{lem}\label{5.5}
Let $X$ be a smooth scheme over $k$, and $P\in X$ a closed point. If $\wt{X}$ is a smooth lifting of $X$ 
over $W_*(k)$, then there exists a closed subscheme $\wt{P}\subset\wt{X}$ such that $\wt{P}$ is smooth 
over $W_*(k)$ and $\wt{P}\times_{\spec W_*(k)}\spec k=P$.
\end{lem}

\begin{proof}
Let $X_n$ be the induced lifting from $\wt{X}$ of $X$ over $W_n(k)$ for any $n<*$ if $*\neq $ null, or for any $n\geq 2$
if $*=$ null. For simplicity, denote $X_1=X$ and $P_1=P$. Assume that there exists a closed subscheme $P_n\subset X_n$
such that $P_n$ is smooth over $W_n(k)$ and $P_n\times_{\spec W_n(k)}\spec W_m(k)=P_m$ for any $m\leq n$.

Consider the following commutative diagram:
\[
\xymatrix{
\spec W_n(k) \ar@{^{(}->}[d] \ar@{^{(}->}[r]^{\qquad P_n} & X_n \ar@{^{(}->}[r] & X_{n+1} \ar[d]^{\eta_{n+1}} \\
\spec W_{n+1}(k) \ar@{=}[rr] \ar@{-->}[urr]^{\xi_{n+1}} & & \spec W_{n+1}(k)
}
\]
Since $\spec W_n(k)\hookrightarrow\spec W_{n+1}(k)$ is a closed immersion with ideal sheaf square zero and
$\eta_{n+1}:X_{n+1}\ra \spec W_{n+1}(k)$ is smooth, there exists a morphism $\xi_{n+1}:\spec W_{n+1}(k)\ra X_{n+1}$
such that the induced diagrams are commutative. Since $\xi_{n+1}$ is a section of $\eta_{n+1}$, it defines a
closed subscheme $P_{n+1}\subset X_{n+1}$ smooth over $W_{n+1}(k)$. By the upper commutativity, we have
$P_n=P_{n+1}\times_{\spec W_{n+1}(k)}\spec W_n(k)$.

If $*\neq$ null, then the conclusion is proved by induction. If $*=$ null, then we have morphisms $\zeta_n:\spec W_n(k)
\stackrel{\xi_n}{\ra} X_n\inj \wt{X}$, hence homomorphisms $\zeta_n^*:\OO_{\wt{X}}\ra W_n(k)$ which are compatible
for any $n\geq 1$. Thus we obtain an induced homomorphism $\zeta^*=\varprojlim\zeta_n^*:\OO_{\wt{X}}\ra \varprojlim W_n(k)=W(k)$,
which defines a closed subscheme $\wt{P}\subset\wt{X}$ such that $\wt{P}$ is smooth over $W(k)$ and
$\wt{P}\times_{\spec W(k)}\spec k=P$.
\end{proof}

\begin{prop}\label{5.6}
Let $X$ be a smooth scheme over $k$ with $\dim X\geq 2$, $P\in X$ a closed point, and $\pi:X'\ra X$ 
the blow-up of $X$ along $P$. If $X$ is strongly liftable over $W_*(k)$, then $X'$ is strongly liftable 
over $W_*(k)$.
\end{prop}

\begin{proof}
It is similar to that of \cite[Proposition 2.6]{xie10} by using Lemmas \ref{5.4} and \ref{5.5}.
\end{proof}

\begin{thm}\label{5.7}
Let $X$ be a smooth projective variety. If $H^2(X,\TT_X)=H^2(X,\OO_X)=0$, then $X$ has a smooth lifting over $W(k)$.
\end{thm}

\begin{proof}
By \cite[II, Theorem 3]{se79}, $W(k)$ is a complete discrete valuation ring with residue field $k$, so the
conclusion follows from \cite[Expos\'e III, Th\'eor\`eme 7.3]{sga1}.
\end{proof}

\begin{thm}\label{5.8}
The following varieties are strongly liftable over $W(k)$:
\begin{itemize}
\item[(i)] $\A^n_k$, $\PP^n_k$ and a smooth projective curve;
\item[(ii)] a smooth projective variety of Picard number 1 which is a complete intersection in $\PP^n_k$.
\end{itemize}
\end{thm}

\begin{proof}
It is similar to those of \cite[Lemmas 3.1--3.3]{xie10}, whereas for a smooth projective curve, its strong liftability
follows from Theorem \ref{5.7} and Lemma \ref{5.5}.
\end{proof}

Let $X$ be a smooth scheme over $k$, $\wt{X}$ a lifting of $X$ over $W(k)$, $\LL$ an invertible sheaf on $X$, and
$\wt{\LL}$ a lifting of $\LL$ on $\wt{X}$. Tensoring (\ref{es16}) by $\wt{\LL}$, we have the following exact sequence:
\begin{eqnarray*}
0\ra \wt{\LL}\stackrel{V}{\ra} \wt{\LL}\stackrel{r}{\ra} \LL\ra 0.
\end{eqnarray*}

There is a criterion for strong liftability over $W(k)$ similar to Proposition \ref{2.3}.

\begin{prop}\label{5.9}
Let $X$ be a smooth scheme over $k$, and $\wt{X}$ a smooth lifting of $X$ over $W(k)$. 
If for any prime divisor $D$ on $X$, there is a lifting $\wt{\LL}_D$ of $\LL_D=\OO_X(D)$ on $\wt{X}$ 
such that the natural map $r:H^0(\wt{X},\wt{\LL}_D)\ra H^0(X,\LL_D)$ is surjective, 
then $X$ is strongly liftable over $W(k)$.
\end{prop}

\begin{proof}
Let $s\in H^0(X,\LL_D)$ be a section corresponding to the prime divisor $D$. By assumption, there is a section
$\wt{s}\in H^0(\wt{X},\wt{\LL}_D)$ with $r(\wt{s})=s$. Let $\wt{D}=\divisor_0(\wt{s})$ be the divisor of zeros
of $\wt{s}$. Then it is easy to see that $\wt{D}\times_{\spec W(k)}\spec k=D$. To prove the flatness of $\wt{D}$
over $W(k)$, note first that $\OO_{\wt{D}}$ is idealwise separated for the maximal ideal of $W(k)$, so by the local
criteria of flatness \cite[(20.C) Theorem 49]{ma80}, it suffices to prove that $\tor^{W(k)}_1(\OO_{\wt{D}},k)=0$.
Considering the following exact sequence:
\begin{eqnarray*}
0\ra \OO_{\wt{X}}(-\wt{D})\ra \OO_{\wt{X}}\ra \OO_{\wt{D}}\ra 0,
\end{eqnarray*}
and taking its long exact sequence for $-\otimes_{W(k)}k$, we obtain an exact sequence:
\begin{eqnarray}
0\ra \tor^{W(k)}_1(\OO_{\wt{D}},k)\ra \OO_{\wt{X}}(-\wt{D})\otimes_{W(k)}k\ra \OO_X\ra \OO_D\ra 0. \label{es18}
\end{eqnarray}
To prove that $\tor^{W(k)}_1(\OO_{\wt{D}},k)=0$, by (\ref{es18}), it suffices to prove that $\OO_{\wt{X}}(-\wt{D})
\otimes_{W(k)}k=\OO_X(-D)$, which is clear from the construction of $\wt{D}$. Thus $\wt{D}$ is flat over $W(k)$,
whence $\wt{D}$ is a lifting of $D$ over $W(k)$.
\end{proof}

\begin{thm}\label{5.10}
A smooth projective toric variety is strongly liftable over $W(k)$.
\end{thm}

\begin{proof}
It is similar to that of \cite[Theorem 3.1]{xie11}, which depends heavily on \cite[Page 66, Lemma]{fu93}.
\end{proof}

\begin{thm}\label{5.11}
A smooth projective rational surface is strongly liftable over $W(k)$.
\end{thm}

\begin{proof}
Let $X$ be a smooth projective rational surface over $k$. If $X\cong\PP^2_k$, then the conclusion was 
already proved in Theorem \ref{5.8}. If $X\not\cong\PP^2_k$, then there is a birational morphism
$X\ra F_n=\PP(\OO_{\PP^1_k}\oplus\OO_{\PP^1_k}(-n))$, which is the composition of a sequence of
blow-ups along closed points. Since the Hirzebruch surface $F_n$ is toric, hence is strongly
liftable over $W(k)$ by Theorem \ref{5.10}. It follows from Proposition \ref{5.6} that $X$ is
strongly liftable over $W(k)$.
\end{proof}

Finally, we give two results on cyclic covers over $W(k)$.

\begin{thm}\label{5.12}
Let $X$ be a smooth projective variety, and $\LL$ an invertible sheaf on $X$. Let $N$ be a positive
integer prime to $p$, and $D$ an effective divisor on $X$ with $\LL^N=\OO_X(D)$ and $\Sing(D_{\rm red})
=\emptyset$. Let $\pi:Y\ra X$ be the cyclic cover obtained by taking the $N$-th root out of $D$.
If $X$ has a lifting $\wt{X}$ over $W(k)$, $\LL$ has a lifting $\wt{\LL}$ on $\wt{X}$, and $D$ has a
lifting $\wt{D}$ on $\wt{X}$ with $\wt{\LL}^N=\OO_{\wt{X}}(\wt{D})$, then $Y$ is a smooth projective
scheme liftable over $W(k)$.
\end{thm}

\begin{proof}
It is similar to that of \cite[Theorem 4.1]{xie11}.
\end{proof}

\begin{cor}\label{5.13}
Let $X$ be a smooth projective toric variety, and $\LL$ an invertible sheaf on $X$.
Let $N$ be a positive integer prime to $p$, and $D$ an effective divisor on $X$ with
$\LL^N=\OO_X(D)$ and $\Sing(D_{\rm red})=\emptyset$.
Let $\pi:Y\ra X$ be the cyclic cover obtained by taking the $N$-th root out of $D$.
Then $Y$ is a smooth projective scheme liftable over $W(k)$.
\end{cor}

\begin{proof}
It is similar to that of \cite[Corollary 4.4]{xie11}.
\end{proof}

By Corollary \ref{5.13}, we can obtain a large class of smooth projective varieties 
liftable over $W(k)$. We give an easy example to illustrate this idea.

\begin{ex}\label{5.14}
Let $X=\PP^n_k$, $\LL=\OO_X(1)$ and $N$ a positive integer such that $n\geq 2$, $(N,p)=1$ and $N>n+2$.
Let $H$ be a general element in the linear system of $\OO_X(N)$. Then $H$ is a smooth irreducible
hypersurface of degree $N$ in $X$ with $\LL^N=\OO_X(H)$. Let $\pi:Y\ra X$ be the cyclic cover
obtained by taking the $N$-th root out of $H$. Then by Corollary \ref{5.13}, $Y$ is a smooth projective 
variety liftable over $W(k)$. By Hurwitz's formula, we have $K_Y=\pi^*(K_X+\frac{N-1}{N}H)$.
Since the degree of $K_X+\frac{N-1}{N}H$ is $N-(n+2)>0$, $K_Y$ is an ample divisor on $Y$,
hence $Y$ is of general type.
\end{ex}

\vskip 10mm

\textsc{Qihong Xie}

\textsc{School of Mathematical Sciences, Fudan University,
Shanghai 200433, China}

\textit{E-mail address}: \texttt{qhxie@fudan.edu.cn}

\vskip 5mm

\textsc{Jian Wu}

\textsc{School of Mathematical Sciences, Fudan University,
Shanghai 200433, China}

\textit{E-mail address}: \texttt{10210180021@fudan.edu.cn}

\end{document}